\documentclass[a4paper,11pt]{article}
\title{On some Generalized Fermat Equations of the form $x^2+y^{2n} = z^p$ }
\author{Philippe Michaud-Jacobs}

\makeatletter
\newcommand\notsotiny{\@setfontsize\notsotiny\@vipt\@viipt}
\makeatother
\usepackage{amsmath} 
\usepackage{amssymb} 
\usepackage{mathtools}
\usepackage{framed}
\usepackage{pifont} 
\usepackage{enumerate} 
\usepackage{wasysym} 
\usepackage{commath}
\usepackage{graphicx}
\usepackage{framed}
\usepackage{amsmath}
\usepackage{amssymb}
\usepackage{amsfonts}
\usepackage{mathrsfs}
\usepackage{mathtools}
\usepackage{url}
\usepackage{array}
\usepackage{amsthm}
\usepackage[english]{babel}
\usepackage[utf8]{inputenc}
\newtheorem{theorem}{Theorem}[section]
\newtheorem*{theorem*}{Theorem}

\newtheorem {corollary}[theorem]{Corollary}
\newtheorem {lemma}[theorem]{Lemma}
\newtheorem {proposition}[theorem]{Proposition}
\theoremstyle{definition}

\theoremstyle{remark}
\newtheorem{remark}[theorem]{Remark}
\newtheorem*{remark*}{Remark}
\usepackage{yfonts}
\usepackage{amsmath,amssymb,tikz-cd}
\usepackage{pdflscape}

\usepackage{etoolbox}
\apptocmd{\sloppy}{\hbadness 10000\relax}{}{}

\usepackage[numbers]{natbib}
\usepackage{cite}
\usepackage{mathtools, nccmath}

\usepackage[section]{placeins}

\providecommand{\Q}{\mathbb{Q}}
\providecommand{\Z}{\mathbb{Z}}

\usepackage{longtable}
\usepackage{hyperref}

\usepackage{multirow}

\newcommand{\Addresses}{{% additional braces for segregating \footnotesize
  \bigskip
  \footnotesize

 \textsc{Mathematics Institute, University of Warwick, CV4 7AL, United Kingdom}\par\nopagebreak
  \textit{E-mail address}: \texttt{p.rodgers@warwick.ac.uk}
}}

\let\svthefootnote\thefootnote
\newcommand\freefootnote[1]{%
  \let\thefootnote\relax%
  \footnotetext{#1}%
  \let\thefootnote\svthefootnote%
}

\date{\vspace{-3ex}}

\begin{document}

\maketitle

\begin{abstract}
The primary aim of this paper is to study the generalized Fermat equation \[ x^2+y^{2n} = z^{3p} \] in coprime integers $x$, $y$, and $z$,  where $n \geq 2$ and $p$ is a fixed prime. Using modularity results over totally real fields and the explicit computation of Hilbert cuspidal eigenforms, we provide a complete resolution of this equation in the case $p=7$, and obtain an asymptotic result for fixed $p$. Additionally, using similar techniques, we solve a second equation, namely $x^{2\ell}+y^{2m} = z^{17}$, for primes $\ell,m \ne 5$.
\end{abstract}

\section{Introduction} 

The Diophantine equation \begin{equation}\label{genferm} x^p+y^q=z^r\end{equation} for integers $p,q,r \geq 2$ is known as the \emph{generalized Fermat equation}. Since Wiles' proof of Fermat's Last Theorem \citep{Wiles} some 25 years ago, it has been the subject of intense study, and has been resolved for many infinite families of integer triples $(p,q,r)$. The \emph{generalized Fermat conjecture}, also known as the \emph{Fermat--Catalan conjecture}, states that there are only finitely many triples of non-zero coprime integer powers $(x^p,y^q,z^r)$ satisfying (\ref{genferm}) with $1/p+1/q+1/r <1$. We refer the reader to \citep{Survey2} for an excellent survey on the generalized Fermat equation, which assumes very little background knowledge. We also refer to \citep{Misc} for all (unconditional) results on this equation prior to 2016, as well as \citep[Theorem 8.7]{Stoll}, \citep[Theorem 1]{AnniSiksek} and \citep[Corollary 8.2]{ext} for (unconditional) results on this equation since 2016.
\freefootnote{\emph{Date}: \date{\today}.}
\freefootnote{\emph{Keywords}: Elliptic curves, Frey curve, level-lowering,  generalized Fermat, Galois representations, irreducibility,  Hilbert modular forms.}
\freefootnote{\emph{MSC2010}: 11D41, 11F80, 11G05, 11F41.}
\freefootnote{The author is supported by an EPSRC studentship and has previously used the name Philippe Michaud-Rodgers.}

The primary aim of this paper is to study the equation \begin{equation}\label{MainEq} x^2+y^{2n} = z^{3p}, \end{equation} for $n \geq 2$ and $p$ a fixed prime.  By \citep[Theorem 1]{Misc}, this equation has no solutions in non-zero coprime integers for $p=2$, $3$, or $5$. Using results on the modularity of elliptic curves over totally real fields, irreducibility of Galois representations, and the explicit computation of Hilbert cuspidal eigenforms, we obtain a complete resolution of this equation in the case $p=7$.

\begingroup
\renewcommand\thetheorem{1}
\begin{theorem}\label{Mainthm7} Let $n \geq 2$. The equation \[x^2+y^{2n} = z^{21}\] has no solutions in non-zero coprime integers $x$, $y$, and $z$. 
\end{theorem}
\endgroup

We also obtain the following asymptotic result.

\begingroup
\renewcommand\thetheorem{2}
\begin{theorem}\label{Mainthm} There exists an effectively computable constant $C(p)$, depending only on the prime $p$, such that for all primes $\ell > C(p)$, the equation \[x^2+y^{2\ell} = z^{3p}\] has no solutions in non-zero coprime integers $x$, $y$, and $z$. 
\end{theorem}
\endgroup

For small values of $p>7$, it is possible to compute such a constant $C(p)$. For example, in Proposition \ref{11} we find that we can take $C(11) =  10^{2930}$.

We start, in Section 2, by stating some known results on equation (\ref{MainEq}) and introducing two technical lemmas. Then, in Section 3, we carry out a descent argument. This is initiated by a factorisation of the left-hand side of (\ref{MainEq}) over the field $\Q(i)$, which leads to new ternary equations over the maximal real subfield of the $p$th cyclotomic field. In Section 4, we associate a family of Frey elliptic curves to these equations, and use standard level-lowering results to relate these curves (or more precisely their Galois representations) to Hilbert cuspidal eigenforms. Crucially, these Frey curves will have multiplicative reduction at the primes above $3$, and this will allow us to circumvent the issues posed by the trivial solutions (those solutions satisfying $xyz=0$). These arguments allow us to prove Theorems \ref{Mainthm7} and \ref{Mainthm} in Sections 5 and 6 respectively. 

Finally, in Section 7 we consider a different, although similar equation, namely \begin{equation}\label{eqp} x^{2\ell}+y^{2m} = z^{p},\end{equation} for primes  $\ell$ and $m$, and $p$ a fixed odd prime.  This equation has no solutions in non-zero coprime integers $x,y,$ and $z$, for $p \in \{ 3,5,7,11 \}$, and for $p=13$ when $\ell,m \ne 7$ (see  \citep[Theorem~1.1]{AnniSiksek}, \citep[Theorem~1]{Misc}, \citep[Theorem~1]{BEN}, and \citep[Theorem~1]{26n}). The case $p=13$ was then completed in \citep[Corollary 8.2]{ext}. We partially extend these results to the case $p=17$. The main difficulty in the case $p=17$ is the impossibility of computing the full Hilbert cusp form data at the required levels. We overcome this by working directly with Hecke operators to prove the following theorem.

\begingroup
\renewcommand\thetheorem{3}
\begin{theorem}\label{Thm3} Let $\ell,m \ne 5$ be primes. The equation \[x^{2\ell}+y^{2m} = z^{17}\] has no solutions in non-zero coprime integers $x$, $y$, and $z$. 
\end{theorem}
\endgroup

By \citep[Theorem 8.7]{Stoll}, the equation $x^5+y^5 = z^{17}$ has no solutions in non-zero coprime integers $x$, $y$, and $z$. Using this we obtain the following corollary to Theorem \ref{Thm3}.

\begingroup
\renewcommand\thetheorem{4}
\begin{corollary}\label{corol4} Let $n \geq 2$. The equation \[x^{2n}+y^{2n} = z^{17}\] has no solutions in non-zero coprime integers $x$, $y$, and $z$. 
\end{corollary}
\endgroup

\bigskip

The \texttt{Magma} \citep{magma} code used to support the computations in this paper can be found at: 

\vspace{3pt}

 \url{https://warwick.ac.uk/fac/sci/maths/people/staff/michaud/c/}

\bigskip

I am extremely grateful to my supervisors Samir Siksek and Damiano Testa for their support in writing this paper. I would also like to thank the anonymous referee for a very careful reading of the paper. 

\section{Known Results and Preliminaries}

If a triple of integers $(x,y,z)$ satisfies (\ref{genferm}), then we shall say that the solution is \emph{non-trivial} if $xyz \ne 0$, and \emph{primitive} if $x,y,$ and $z$ are coprime.

We start by stating what we can deduce about solutions to (\ref{MainEq}) from other results on generalized Fermat equations. 

\begin{theorem}[\citep{BEN}, \citep{refined}, \citep{Misc}]\label{known} Let $n \geq 2$ and let $p$ be prime. Suppose there exist non-zero coprime integers $x$, $y$, and $z$ satisfying \begin{equation*}x^2+y^{2n} = z^{3p}.\end{equation*} Then $ n > 10^7$, $p>5$, $y \equiv 3 \pmod{6}$, $x$ is even, and $z$ is odd.
\end{theorem}

\begin{proof} If $n=2$, then there are no non-trivial primitive solutions to (\ref{MainEq}) by \citep[Theorem 1]{BEN}, so we will suppose $n>2$. If $p = 2$, $3$, or $5$, then there are no non-trivial primitive solutions by \citep[Theorem 1]{Misc}. Next, we have that $n >10^7$ and $y \equiv 3 \pmod{6}$ by \citep[p.~11]{Misc}. Finally, since $y$ is odd, we see that $x$ is even and $z$ is odd by considering the equation modulo $4$.
\end{proof}

We note that the equation $x^2+y^{2n}=z^3$ admits the trivial solutions $(\pm 1,0,1)$ and $(0,\pm 1,1)$. The trivial solution $(0,\pm 1,1)$ would usually render unfeasible a successful application of the modular method. The reason that this trivial solution can be ruled out in this case is because the corresponding Frey curve (which is defined over $\Q$) at this solution has complex multiplication. Indeed, following the arguments of \citep[p.~1306]{refined} and the proof of Theorem 1.1 in \citep{pp2} in the case $C = 3$, one finds that the Frey curve at this solution is an elliptic curve of conductor $32$ with CM field $\Q(\sqrt{-1})$. 

By Theorem \ref{known}, we can restrict to the case $n  = \ell$, prime, with $\ell > 10^7$.

\begin{proposition}\label{pmidy} Let $\ell,p \geq 5$ be primes.  Suppose there exist non-zero coprime integers $x$, $y$, and $z$ satisfying \begin{equation*}x^2+y^{2\ell} = z^{3}.\end{equation*} If $p \mid y$, then $\ell < (\sqrt{p}+1)^2$.
\end{proposition}

\begin{proof} By \citep[pp.~1306--1307]{refined}, there exist coprime integers $u$ and $v$, with $uv \ne 0$, $3 \mid v$, $u$ even, and $v$ odd, such that \begin{equation}\label{yuv} y^\ell = v(3u^2-v^2). \end{equation} We associate to (\ref{yuv}) the Frey elliptic curve \[W: \; Y^2 = X^3 + 2uX^2+v^2, \] which has minimal discriminant and conductor \[ \Delta_\mathrm{min} =  2^6 \cdot 3^{-3} \cdot v^4(3u^2-v^2), \qquad N = 2^5 \cdot 3 \cdot \mathrm{Rad}_{2,3}(\Delta_\mathrm{min}).\] Here, $\mathrm{Rad}_{2,3}(\Delta_\mathrm{min})$ denotes the product of all primes other than $2$ or $3$ dividing $\Delta_\mathrm{min}$. 

Still following \citep[pp.~1306--1307]{refined}, we level-lower the curve $W$, and find that $\overline{\rho}_{W,\ell} \sim \overline{\rho}_{W_0,\ell}$, for $W_0$ an elliptic curve of conductor $96 = 2^5 \cdot 3$. Now, if $p \mid y$, then $p \mid y^\ell = v(3u^2-v^2)$, so $p \mid \Delta_\mathrm{min}$ and $W$ has multiplicative reduction at $p$. Also $p \nmid 96$ as $p \geq 5$, so \[ \ell \mid p+1 + a_p(W_0) \quad \text{or} \quad \ell \mid p+1 - a_p(W_0). \] Then $\abs{a_p(W_0)} \leq 2 \sqrt{p}$, so \[ \ell < p+1+2\sqrt{p} = (\sqrt{p}+1)^2, \] as required. \end{proof}

In order to prove Theorems \ref{Mainthm7} and \ref{Mainthm}, we will start (in Section 3) by carrying out a descent argument over the maximal real subfield of the $p$th cyclotomic field. In this section, we introduce some notation as well as two lemmas that will be useful in the sequel.

Let $p$ be an odd prime. We write $\zeta_p$ for a primitive $p$th root of unity, so that $\Q(\zeta_p)$ is the $p$th cyclotomic field which has degree $p-1$. We write $K=\Q(\zeta_p+\zeta_p^{-1})$ for the maximal real subfield of $\Q(\zeta_p)$. The field $K$ is a totally real abelian Galois field of degree $(p-1)/2$. We write $\mathcal{O}_K$ for the ring of integers of the field $K$. Then $\mathcal{O}_K = \Z[\zeta_p+\zeta_p^{-1}]$. The prime $p$ is totally ramified in $K$, and we write $\mathfrak{p}$ for the unique prime ideal of $\mathcal{O}_K$ above $p$. We have \[ p\mathcal{O}_K =  \mathfrak{p}^{(p-1)/2}. \] More generally we will denote prime ideals of $\mathcal{O}_K$ by $\mathfrak{q}$, or sometimes by $\mathfrak{q}_m$ for a prime above the rational prime $m \in \Z$.
We also introduce the notation \[ \theta_j \coloneqq \zeta_p^j+\zeta_p^{-j}, \quad \text{for}~ j = 1, \dotsc, (p-1)/2. \] For further background on cyclotomic fields and their subfields, we refer to \citep[pp.~1--19]{Washington}.

\begin{lemma}[{\citep[Lemma 3.1]{AnniSiksek}}] For $1 \leq j \leq (p-1)/2$ we have \[ \theta_j, \, \theta_j+2 \in \mathcal{O}_K^\times \quad  \text{and} \quad (\theta_j-2)\mathcal{O}_K = \mathfrak{p}.\] For $1 \leq j < k \leq (p-1)/2$ we have \[(\theta_j-\theta_k)\mathcal{O}_K = \mathfrak{p}. \]
\end{lemma}

\begin{lemma}\label{binom}
For any $m \geq 1$ and $1 \leq j \leq (p-1)/2$ we have \[ \theta_j^{2^{(p-1)m}} \equiv \theta_j + 2 \pmod{4 \mathcal{O}_K}.\]
\end{lemma}

\begin{proof}
Write $r=(p-1)m \geq 2$. Then \[ \theta_j^{2^r} = (\zeta_p^j+\zeta_p^{-j})^{2^r} = \sum_{i=0}^{2^r} \binom{2^r}{i} \zeta_p^{ji}\zeta_p^{-j(2^r-i)} ~.  \] Using Legendre's formula for the prime decomposition of a factorial, we have that $v_2 \binom{2^t}{2^{t-1}} = 1 $ for any $t \geq 1$.
From this, and the identity \[ \binom{2^r}{i} = \sum_{t=0}^i \binom{2^{r-1}}{t} \binom{2^{r-1}}{2^{r-1}-t} ~, \] it is straightforward to show by induction on $r$ that \[ v_2 \binom{2^r}{i} \geq 2, \quad \text{for} ~ 0 < i < 2^r, i \ne 2^{r-1} ~.\]
Then \[ \sum_{i=0}^{2^r} \binom{2^r}{i} \zeta_p^{ji}\zeta_p^{-j(2^r-i)} \equiv (\zeta_p^j)^{2^{m(p-1)}} + 2 + (\zeta_p^{-j})^{2^{m(p-1)}}  \equiv \theta_j + 2 \pmod{4\mathcal{O}_K}~,\] with the last equivalence coming from the fact that $2^{m(p-1)} \equiv 1 \pmod{p}$.
\end{proof}

\section{Descent}

Suppose there exist coprime integers $x$, $y$, and $z$, satisfying \begin{equation}\label{maineq} x^2+y^{2\ell}=z^{3p} \end{equation} for primes $\ell > 3$ and $p>5$. We wish to obtain a factorisation for $y^\ell$ over the field $K$. We follow the descent argument of \citep[pp.~1154--1155]{AnniSiksek}. We  start by considering the following factorisation over  $\Q(i)$: \[ (y^{\ell}+xi)(y^{\ell}-xi) = {(z^3)}^p. \] Since $x$ and $y$ are coprime, there exist $a,b \in \Z$ such that \[y^\ell + xi = (a+bi)^p \quad \text{and} \quad z^3 = a^2+b^2. \] Comparing real and imaginary parts, we obtain \begin{equation}\label{yl} y^\ell = \frac{(a+bi)^p + (a-bi)^p}{2} \;. \end{equation} Since $y$ and $z$ are coprime, we see that $a$ and $b$ are also coprime. 

We recall the standard factorisation, for $u,v \in \mathbb{C}$, \[u^p+v^p = \prod_{j=0}^{p-1} ( u + v \zeta_p^j ) = (u+v) \prod_{j=1}^{(p-1)/2} (u+v \zeta_p^j)(u+v \zeta_p^{-j}) \;. \]
Applying this to (\ref{yl}), we obtain 
\begin{align*}
y^\ell 
&= a \cdot \prod_{j=1}^{(p-1)/2}
\left(
(a+bi)+(a-bi) \zeta_p^j
\right) \cdot
\left(
(a+bi)+(a-bi) \zeta_p^{-j}
\right) \\
& = a \cdot \prod_{j=1}^{(p-1)/2}
\left(
(\theta_j+2) a^2+ (\theta_j-2) b^2
\right)    \, .
\end{align*} 
So \begin{equation}\label{Facty}
 y^\ell  = a \cdot \prod_{j=1}^{(p-1)/2} \beta_j \;,  \end{equation} where
\[ \beta_j \coloneqq (\theta_j+2) a^2 + (\theta_j-2)b^2, \qquad \text{ for}~~ j=1,\dotsc,(p-1)/2 \, . \]
From (\ref{Facty}), we see that $a$ is odd (since $y$ is odd), and so $b$ is even.

By Theorem \ref{known}, we know that $3 \mid y$. We now claim that $3 \mid a$. Suppose not. If $3 \nmid b$, then $z^3 = a^2 + b^2 \equiv -1 \pmod{3} $, so $z \equiv -1 \pmod{3}$, a contradiction by reducing (\ref{maineq}) mod $3$. So $3 \mid b$. Write $\mathfrak{q}_3$ for a prime of $K$ above $3$. Then since $3 \mid y$ but  $3 \nmid a$, we have that $\mathfrak{q}_3 \mid \beta_j$ for some $j \in \{1, \dotsc, (p-1)/2 \}$. So  \[ \mathfrak{q}_3 \mid \beta_j - (\theta_j -2)b^2 = (\theta_j+2)a^2  . \] So $\mathfrak{q} \mid (\theta_j+2) \in \mathcal{O}_K^{\times}$, a contradiction. We conclude that $3 \mid a$.

\begin{lemma}\label{faclemma} Suppose $p \nmid y$. Then 
\[
a= \alpha^\ell,
\qquad 
\beta_j \mathcal{O}_K=\mathfrak{b}_j^\ell \, ,
\]
where $\alpha \in \Z$ with $\alpha \equiv 3 \pmod{6}$, and
$\alpha \mathcal{O}_K,\mathfrak{b}_1,\dotsc,\mathfrak{b}_{(p-1)/2}$ are pairwise coprime ideals of $\mathcal{O}_K$, all coprime to $2p$.
\end{lemma}

\begin{proof} We follow the first part of the proof of \citep[Lemma~4.1]{AnniSiksek}.
Since $ 2 \mid b$ and $2 \nmid a$, we see that the $\beta_j$ are coprime to $2 \mathcal{O}_K$. Let $\mathfrak{q}$ be a prime of $K$ and suppose $\mathfrak{q}$ divides $a$ and $\beta_j$. Then it also divides $(\theta_j-2)b^2$, and since $a$ and $b$ are coprime, it divides $(\theta_j-2) \mathcal{O}_K = \mathfrak{p}$. So $\mathfrak{q}=\mathfrak{p}$, a contradiction, since $p \nmid y$.

Next, suppose $\mathfrak{q}$ is a prime of $K$ with $\mathfrak{q} \mid \beta_j, \beta_k$ for $j \ne k$. Then \begin{align*} \mathfrak{q} \mid & (\theta_k -2)\beta_j - (\theta_j-2)\beta_k = ((\theta_j+2)(\theta_k-2)-(\theta_k+2)(\theta_j-2))a^2, \\
\mathfrak{q} \mid & (\theta_j+2)\beta_k  -(\theta_k +2)\beta_j = ((\theta_j+2)(\theta_k-2)-(\theta_k+2)(\theta_j-2))b^2. 
\end{align*}
Since $a$ and $b$ are coprime, we see that \[ \mathfrak{q} \mid (\theta_j+2)(\theta_k-2)-(\theta_k+2)(\theta_j-2) = 4(\theta_k-\theta_j). \] Since $ (\theta_k-\theta_j) \mathcal{O}_K =  \mathfrak{p}$ and $\beta_j$ is coprime to $2\mathcal{O}_K$, we have $\mathfrak{q} = \mathfrak{p}$, another contradiction.
So the ideals $a \mathcal{O}_K$ and $\beta_j \mathcal{O}_K$ are pairwise coprime, and also all coprime to $2p$. The lemma follows.
\end{proof}

\section{Frey Curves}

We will now associate a Frey curve (in fact a family of Frey curves) to (\ref{Facty}) when $p \nmid y$. The key difference between the Frey curve we define compared to the one defined in \citep[p.~1156]{AnniSiksek} is its behaviour at the primes of $K$ above $2$. The Frey curve we define will have additive, rather than multiplicative, reduction at the primes above $2$, and is therefore not semistable. The main consequences of this are that we will need to apply different modularity and irreducibility results in Sections 5 and 6, and it will also limit our ability to compute Hilbert cusp forms.

Suppose $p \nmid y$. We now fix $j$ and $k$ such that $1 \leq j< k \leq (p-1)/2$. Let 
 \begin{equation}\label{uvw}
u= 
\beta_j ,  \qquad
v=- \frac{(\theta_j-2)}{(\theta_k-2)} \cdot \beta_k , \qquad
w=\frac{4 (\theta_j-\theta_k)}{(\theta_k-2)} \cdot a^2.
\end{equation}
Then
$u+v+w=0$, and by Lemma \ref{faclemma} we have
\[
u \mathcal{O}_K=\mathfrak{b}_j^\ell, \qquad
v \mathcal{O}_K=\mathfrak{b}_k^\ell, \qquad
w \mathcal{O}_K= 4 \cdot \alpha^{2 \ell} \cdot \mathcal{O}_K.
\]
We define the Frey elliptic curve
\begin{equation*}
E=E_{j,k} : \; Y^2=X(X-v)(X+w).
\end{equation*}
We note that $u,v,$ and $w$, are defined as in \citep[p.~1156]{AnniSiksek}, but the Frey curve we have chosen differs. We discuss this choice in Remark \ref{twistrem}. 

Write $\mathrm{Rad}(\mathfrak{c})$ to denote the product of prime ideals dividing a non-zero ideal $\mathfrak{c}$ of $\mathcal{O}_K$. 

\begin{lemma}\label{Frey1}
The curve $E$ has good reduction at $\mathfrak{p}$ and multiplicative reduction at all primes of $K$ above $3$. It has minimal discriminant and conductor
\[
\mathcal{D}
=
2^8 {\alpha}^{4 \ell} \mathfrak{b}_j^{2\ell} \mathfrak{b}_k^{2\ell}, \qquad
\mathcal{N} = 2^3 \cdot \mathrm{Rad}(\alpha \mathfrak{b}_j \mathfrak{b}_k).
\]
\end{lemma}

\begin{proof}
We have $\Delta = 16u^2v^2w^2$ and $c_4 = 16(w^2-uv)$. We see that $\mathfrak{p} \nmid \Delta$, so $E$ has good reduction at $\mathfrak{p}$. By Lemma \ref{faclemma}, $c_4$ and $\Delta$ are coprime away from $2$, so the Frey curve is semistable away from $2$. The curve $E$ has multiplicative reduction at all primes above $3$ because  $3 \mid \alpha$ by Lemma \ref{faclemma}. 

Let $\mathfrak{q}$ be a prime of $K$ above $2$. We note that the model is minimal at $\mathfrak{q}$ since $v_\mathfrak{q}(\Delta)=8<12$. So $v_\mathfrak{q}(\mathcal{D}) = 8$, and it remains to show that we have $v_\mathfrak{q}(\mathcal{N}) = 3$. We do this using Tate's algorithm \citep{Tatealg}. We follow the exposition of Tate's algorithm in \citep[pp.~364--368]{Silverman} and outline the main steps.

Since $v_\mathfrak{q}(2) =1$, we can take $2$ as a uniformiser for the local field $K_\mathfrak{q}$. Write $k$ for the residue field $\mathcal{O}_K / \mathfrak{q}$.  Now, $\mathfrak{q}^2 \mid w$, so the point $(\tilde{0},\tilde{0})$ is a singular point of $E / k$, so $\mathfrak{q} \mid a_3,a_4,a_6$. We then find that $\mathfrak{q} \mid b_2, \, \mathfrak{q}^2 \mid a_6,$ and $\mathfrak{q}^3 \mid b_8$, so we proceed directly to Step 6. 

Note that $a_2 \equiv v \equiv \theta_j \pmod{\mathfrak{q}}$. We would like to choose $\gamma$ such that $\gamma^2 \equiv \theta_j \pmod{\mathfrak{q}}$. Write $f$ for the inertia degree of $\mathfrak{q}$. We choose  \[ \gamma = \theta_j^{2^{(p-1)f-1}}. \] Then $\gamma^2 = \theta_j^{2^{f(p-1)}}  \equiv \theta_j \pmod{\mathfrak{q}}$.  We then apply the transformation $Y \mapsto Y-\gamma X$ to obtain \[E': \; Y^2 - 2 \gamma XY = X(X-v)(X+w) - \gamma^2 X^2. \]
We denote the Weierstrass coefficients of $E'$ by $a_i'$. Continuing with Step 6, we consider the polynomial \[ P(T) \coloneqq T^3 + \frac{a_2'}{2}T^2+\frac{a_4'}{2^2}T + \frac{a_6'}{2^3} =  T \left( T^2+\frac{a_2'}{2}T+\frac{a_4'}{2^2} \right). \] Here, $a_4' = -vw \not\equiv 0 \pmod{\mathfrak{q}^3}$, so $P$ does not have a triple root in $\overline{k}$, and we continue to Step 7. We claim that $\mathfrak{q}^2 \mid a_2' = -v + w -\gamma^2$, so that $P$ has a double root in $\overline{k}$. Since $a$ is odd and $b$ is even, $a^2 \equiv 1 \pmod{\mathfrak{q}}$ and $b^2 \equiv 0 \pmod{\mathfrak{q}^2}$. So \begin{align*} -v+w-\theta_j^{2^{(p-1)f}} & \equiv \frac{(\theta_j-2)(\theta_k+2)}{(\theta_k-2)}-\theta_j^{2^{(p-1)f}} \pmod{\mathfrak{q}^2} \\ & \equiv \theta_j+2-\theta_j^{2^{(p-1)f}} \pmod{\mathfrak{q}^2} \\ & \equiv 0 \pmod{\mathfrak{q}^2} \, , \end{align*} where we have applied Lemma \ref{binom} in the final step.

We now start the subprocedure of Step 7 by choosing $\varphi \in \mathcal{O}_K$ such that $\varphi^2 \equiv -vw/2^2 \pmod{\mathfrak{q}}$. We note that $\mathfrak{q} \nmid \varphi$.
We apply the transformation $X \mapsto X+2\varphi$ and denote our new Weierstrass coefficients by $a_i''$. We verify that our new polynomial $P(T)$ (defined as above) now has a double root at $\tilde{0}$. We have that $a_3'' =  -4\gamma \varphi$, and $v_\mathfrak{q}(-4\gamma \varphi)=2,$ so the polynomial \[ Y^2 + \frac{a_3''}{2^2}Y+\frac{a_6''}{2^4} \] has distinct roots in $\overline{k}$, concluding our application of Tate's algorithm. We read off that $v_\mathfrak{q}(\mathcal{N}) = v_\mathfrak{q}(\Delta) - 5 = 3$, with the reduction type at $\mathfrak{q}$ given by the Kodaira symbol $\mathrm{I}_1^*$.
\end{proof}

We note that for a fixed value of $p$, it is possible to verify whether $E$ has split or non-split multiplicative reduction at the primes above $3$. This is because $a^2 \equiv 0 \pmod{3}$ and $b^2 \equiv 1 \pmod{3}$, so if $\mathfrak{q}_3$ denotes a prime of $K$ above $3$, we find that \[c_6 = -32(v-w)(w-u)(u-v) \equiv (\theta_j+1)^3 \pmod{\mathfrak{q}_3}. \] The curve $E$ has split multiplicative reduction at $\mathfrak{q}_3$ if and only if $-c_6$ is a square $\pmod{\mathfrak{q}_3}$ (see \citep[pp.~442--444]{Silverman}). For example, $E$ has split multiplicative reduction (for each choice of $j$) at the unique prime above $3$ when $p=7$, but non-split multiplicative reduction (for each choice of $j$) at the unique prime above $3$ when $p=11$.

\begin{remark}\label{twistrem} In order to simplify the computations in Sections 5 and 6, we would like the conductor of $E$ to be as small as possible. In particular, if we let $\mathfrak{q}$ be a prime of $K$ above $2$, then we would like to minimise $v_\mathfrak{q}(\mathcal{N})$. The best we can hope for would be to decrease this valuation from $3$ to $2$. We cannot decrease this valuation further, as $E$ has potential good reduction at $\mathfrak{q}$.  Unfortunately, we found that by twisting $E$ by units and permuting $u,v$, and $w$, that we could only increase $v_\mathfrak{q}(\mathcal{N})$ to $4$. The curve $E$ we have chosen satisfies $v_\mathfrak{q}(\mathcal{N})=3$ and allows for the easiest application of Tate's algorithm.
\end{remark}

\section{Asymptotic Results}

We would like to apply a suitable level-lowering result to the Frey curve $E$, and combine this with Proposition \ref{pmidy} in order to conclude that any primitive solution $(x,y,z)$ to (\ref{maineq}) is trivial, at least for $\ell$ large enough. We first fix the following notation. We will write $\mathfrak{f}$ for a Hilbert cuspidal eigenform over $K$ of parallel weight $2$, and denote by $\Q_\mathfrak{f}$ its Hecke eigenfield (the field generated by its eigenvalues under the action of the Hecke operators). If $\mathfrak{f}$ is new at its level, then we will simply refer to $\mathfrak{f}$ as a \emph{Hilbert newform}. 

\begin{lemma}\label{levellower}
Let $E$ be the Frey curve defined in Section 4. Suppose $E$ is modular and that $\overline{\rho}_{E,\ell}$
is irreducible. Then
$\overline{\rho}_{E,\ell}\sim \overline{\rho}_{\mathfrak{f},\lambda}$
for a Hilbert newform $\mathfrak{f}$ at level $\mathcal{N}_{\ell}$, where
\[
\mathcal{N}_{\ell}  = 2^3 \cdot \mathcal{O}_K  \, ,
\]
and $\lambda \mid \ell$ is a prime of $\Q_{\mathfrak{f}}$.
\end{lemma}

\begin{proof} We apply \citep[Theorem 7]{asym} to the curve $E$, which is the standard level-lowering result for elliptic curves defined over totally real fields. The statement follows from Lemma \ref{Frey1}.
\end{proof}

Using this, we can prove Theorem \ref{Mainthm}.

\begin{proof}[Proof of Theorem \ref{Mainthm}]
We suppose $(x,y,z)$ is a non-trivial primitive solution to (\ref{maineq}). Suppose $p \nmid y$. Let $E$ denote the Frey curve, defined in Section 4, associated to the solution $(x,y,z)$. Since $K$ is a totally real abelian number field in which $3$ is unramified, and $E$ has semistable reduction (in fact multiplicative reduction) at all primes above $3$, we know that $E$ is modular by \citep[Theorem 1.3]{recipes}.

Next, as $K$ is a totally real Galois field and $E$ is semistable away from $2$, we can apply Theorem 2 of \citep{irred}. Write $B_p$ for the non-zero constant, which depends only on $p$, defined in Theorem 1 of \citep{irred}. Then if $\ell \nmid p \cdot B_p$ (we include a factor of $p$, as $p$ is the only prime that ramifies in K), then $\overline{\rho}_{E,\ell}$ is irreducible for $\ell > (1+3^{3h(p-1)})^2,$ where $h$ denotes the class number of $K$. Since $(1+3^{3h(p-1)})^2>p$, it follows that  $\overline{\rho}_{E,\ell}$ is irreducible for $ \ell > C'(p)$, where\[C'(p) \coloneqq B_p \cdot (1+3^{3h(p-1)})^2 \, . \] 

Suppose $\ell > C'(p)$. Then applying Lemma \ref{levellower}, we have $\overline{\rho}_{E,\ell}\sim \overline{\rho}_{\mathfrak{f},\lambda}$, for a Hilbert newform $\mathfrak{f}$ at level $\mathcal{N}_{\ell}$, and  $\lambda \mid \ell$ a prime of $\Q_{\mathfrak{f}}$. We write $d$ for the dimension of the space of Hilbert cusp forms that are new at level $\mathcal{N}_{\ell}$. Let $\mathfrak{q}_3$ denote a prime of $K$ above $3$. Then $E$ has multiplicative reduction at $\mathfrak{q}_3$ by Lemma \ref{Frey1}. Write $a_{\mathfrak{q}_3}$ for the trace of Frobenius of $\overline{\rho}_{\mathfrak{f},\lambda}$ at $\mathfrak{q}_3$. Then
 \[ \lambda \mid \mathrm{Norm}(\mathfrak{q}_3)+1 + a_{\mathfrak{q}_3}(\mathfrak{f}) \quad \text{or} \quad \lambda \mid \mathrm{Norm}(\mathfrak{q}_3)+1 - a_{\mathfrak{q}_3}(\mathfrak{f}). \] It follows that  \begin{align*} \ell & \mid  \mathrm{Norm}_{\Q_{\mathfrak{f}}/\Q} \left( \mathrm{Norm}(\mathfrak{q}_3)+1 + a_{\mathfrak{q}_3}(\mathfrak{f}) \right) ~ \text{or} \\ \ell & \mid \mathrm{Norm}_{\Q_{\mathfrak{f}}/\Q} \left( \mathrm{Norm}(\mathfrak{q}_3)+1 - a_{\mathfrak{q}_3}(\mathfrak{f})\right) . \end{align*} The size of $a_{\mathfrak{q}_3}(\mathfrak{f})$ is bounded by $2 \sqrt{\mathrm{Norm}(\mathfrak{q}_3)}$, and since $[ \Q_{\mathfrak{f}} : \Q ] < d$, we have \[ \ell \leq \left(\mathrm{Norm}(\mathfrak{q}_3) + 1 + 2 \sqrt{\mathrm{Norm}(\mathfrak{q}_3)}\right)^d = (\sqrt{\mathrm{Norm}(\mathfrak{q}_3)} +1)^{2d} . \]
We set $C(p) = \max\{(C'(p),(\sqrt{\mathrm{Norm}(\mathfrak{q}_3)} +1)^{2d} \}$. If, instead, $p \mid y$, then  $\ell < (\sqrt{p}+1)^2 < C'(p)$ by Proposition \ref{pmidy}. 

We conclude that if $\ell > C(p)$ then we have a contradiciton, so no such non-trivial primitive solution exists.
\end{proof}

Although the constant $C(p)$ in Theorem \ref{Mainthm} is effectively computable, actually computing it is another matter. In the case $p=7$, we are able to compute the Hilbert newforms (using \texttt{Magma}'s Hilbert modular form package) at the level $\mathcal{N}_{\ell}$ and this allows us to compute a (relatively) small value for $C(7)$. Combining this with the fact that we have no solutions for $\ell < 10^7$ will allow us to prove Theorem \ref{Mainthm7} in the next section. Unfortunately, we were unable to compute the Hilbert newforms at level $\mathcal{N}_{\ell}$ for $p >7$. 

When $p \equiv 1 \pmod{4}$, it is possible to choose $j$ and $k$ appropriately and twist the Frey curve $E_{j,k}$ (as in \citep[p.~1157--1158]{AnniSiksek}) so that $E_{j,k}$ is defined over a subfield of $K$, but we were still unable to compute the Hilbert newforms (at the new required level) for $p=13$ (or any larger $p \equiv 1 \pmod{4}$).

Even though we cannot compute the required Hilbert newforms for $p>7$, we can still (following the proof of Theorem \ref{Mainthm}) compute a value $C(p)$, provided we can bound the dimensions of the spaces of Hilbert cusp forms that are new at the level $\mathcal{N}_\ell$. We consider the cases $p=11$, $13$, and $17$. 

\begin{proposition}\label{11}  Let $p=11$, $13$, or $17$. Suppose  $\ell >  C(p)$, with $\ell$ prime, where \[C(11) = 10^{2930}, \qquad C(13) = 10^{90946}, \qquad C(17) = 10^{160315410} . \] Then the equation \[x^2+y^{2\ell} = z^{3p}\] has  no solutions in non-zero coprime integers $x$, $y$, and $z$. 
\end{proposition}

\begin{proof} We follow the proof of Theorem \ref{Mainthm}, computing explicit constants. We first compute  the quantity $B_p$. We find that \[ B_{11}=1, \qquad B_{13} = 2^{18} \cdot 3^{12} \cdot 5^6 \cdot 13^3, \qquad B_{17} = 2^{32} \cdot 5^8 \cdot 13^8 \cdot 17^4 \cdot 67^8.\] Since $\ell >10^7$, we can safely ignore the contribution from $\ell \mid p \cdot B_p$. Since $K$ has class number $1$ in each case, we set $ C'(p) = (1+3^{3(p-1)})^2$. Next, $3$ is inert in $K$ in each case, and $\mathrm{Norm}(3 \cdot \mathcal{O}_K) = 3^{(p-1)/2}$. The dimensions $d$ of the spaces of Hilbert cusp forms that are new at level $\mathcal{N}_{\ell}$ can be computed directly with \texttt{Magma}, and are $1201$, $31422$, and $41883752$, for $p=11$, $13$, and $17$ respectively. We set \[  C(p) = \max\left\lbrace (1+3^{3(p-1)})^2,\left(\sqrt{3^{(p-1)/2}} +1 \right)^{2d} \right\rbrace = \left( \sqrt{3^{(p-1)/2}} +1 \right)^{2d}, \] and the proposition follows. 
\end{proof}

As discussed after the proof of Theorem \ref{Mainthm}, for $p=13$ and $p=17$, we could work over a subfield of $K$ and obtain smaller (although still very large) constants in the above proposition.

It would be interesting to see if it is possible to find a bound on the dimension of the space of Hilbert cups forms that are new at level $\mathcal{N}_{\ell}$ in terms of $p$, or quantities associated to $p$. In this way, it would be possible to obtain a value for the constant $C(p)$ without the need for calculating the dimension explicitly.

\section{The Equation \texorpdfstring{$x^2+y^{2n}=z^{21}$}{}}

We now set $p=7$. The field $K = \Q(\zeta_7+\zeta_7^{-1})$ has degree $3$.
We would first like to prove the irreducibility of $\overline{\rho}_{E,\ell}$ for $\ell> 10^7$.
As the curve $E$ is not semistable, we cannot use the same techniques as in \citep[pp.~1160--1166]{AnniSiksek}.

\begin{lemma}\label{irred7}
Let $p=7$. Let $E$ be the Frey curve defined in Section 4. Then $\overline{\rho}_{E,\ell}$ is irreducible for $\ell > 65 \cdot 6^6$.
\end{lemma} 

\begin{proof} The prime $3$ is inert in $K$, and by Lemma \ref{Frey1}, $E$ has multiplicative reduction at $3\mathcal{O}_K$. Since $3 > \deg(K) -1 = 2$, we can apply Theorem 1.3 of \citep{NajTurc} to deduce that the representation $\overline{\rho}_{E,\ell}$ is irreducible for $\ell > 65 \cdot 6^6$.
\end{proof}

We note that $65\cdot 6^6 <  10^7$, so $\overline{\rho}_{E,\ell}$ is irreducible for $\ell >10^7$.

\begin{proof}[Proof of Theorem \ref{Mainthm7}]
By Theorem \ref{known}, we may restrict to the case of $n=\ell$ prime, with $\ell > 10^7$. We suppose $(x,y,z)$ is a non-trivial primitive solution to (\ref{MainEq}). If $7 \mid y$, then $\ell \leq 13$ by Proposition \ref{pmidy}, so we will assume that $7 \nmid y$, and 
associate the Frey curve $E$ to this solution, as in Section 4. 

The curve $E$ is modular by \citep[Theorem 1.3]{recipes}, or alternatively by applying the more general result that any elliptic curve defined over a totally real cubic field is modular \citep[Theorem 1]{Cubicmod}. By Lemma \ref{irred7}, $\overline{\rho}_{E,\ell}$ is irreducible, and we can therefore apply Lemma \ref{levellower} and level-lower. We have $\overline{\rho}_{E,\ell}\sim \overline{\rho}_{\mathfrak{f},\lambda}$, for a Hilbert newform $\mathfrak{f}$ at level $\mathcal{N}_{\ell}$, and  $\lambda \mid \ell$ a prime of $\Q_{\mathfrak{f}}$. The prime $3$ is inert in $K$, and we write $\mathfrak{q}_3 = 3 \cdot \mathcal{O}_K$, which has norm $27$. 

The dimension of the space of cusp forms that are new a level $\mathcal{N}_\ell$ is $5$. We note that using this information alone is not enough to obtain a contradiction, as the bound obtained following the proof of Theorem \ref{Mainthm} is $(\sqrt{27}+1)^{10} > 10^7$. Instead, we compute the newform decomposition using \texttt{Magma}, and find there are five newforms at level $\mathcal{N}_\ell$ (each with $\Q_\mathfrak{f} = \Q$ necessarily). We can now mimic the proof of Theorem \ref{Mainthm} to obtain the bound  $\ell < (\sqrt{27}+1)^2<39$, giving the desired contradiction. \end{proof}

We note that explicitly computing the values $a_{\mathfrak{q}_3}(\mathfrak{f})$ for each of the five newforms at level $\mathcal{N}_\ell$ would allow us to obtain a sharper bound than $\ell < 39$ in the final step of the above proof, but since we are assuming $\ell > 10^7$ anyway, this is not necessary.

It is in fact possible to avoid the newform computation in the proof of Theorem \ref{Mainthm7}. Setting $a=1$ and $b=0$ (which corresponds to the trivial solution $(0,1,1)$), the Frey curve $E$ is an elliptic curve with conductor $\mathcal{N}_\ell$. By modularity, we obtain a Hilbert newform $\mathfrak{f}$ at level $\mathcal{N}_\ell$ with $\Q_\mathfrak{f} = \Q$. In particular, following the proof of Theorem \ref{Mainthm}, we obtain the improved inequality $\ell < (\sqrt{27}+1)^8 < 10^7$.

\section{The Equation \texorpdfstring{$x^{2\ell}+y^{2m}=z^{17}$}{}}

We now consider the equation \begin{equation}\label{eq17} x^{2\ell}+y^{2m}=z^{17} \end{equation} for primes $\ell$ and $m$. Our aim is to prove Theorem \ref{Thm3}. We directly extend the work carried out in \citep{AnniSiksek}, and so we do not provide a very detailed exposition when the same ideas are present. We continue using the same notation as in the previous sections.

We suppose that $(x,y,z)$ is a primitive solution to (\ref{eq17}). We can interchange $x$ and $y$ to ensure that $x$ is even. This is a key step, as it means that the only trivial solutions are $(0,\pm 1,1)$. When the values corresponding to these trivial solutions are substituted into the Frey curve $F_1$ we define below, we will obtain a singular elliptic curve, and this will not endanger the success of the modular method.
 If $\ell = 2$ then there are no non-trivial primitive solutions by \citep[Theorem 1]{BEN}. If $\ell = 3$ then there are no non-trivial primitive solutions by \citep[Theorem 1]{26n}. If $\ell = 17$ then there are no non-trivial primitive solutions by \citep[Main Theorem]{ll2}. We therefore suppose $\ell \geq 5$ and $\ell \ne 17$. 

As in Section 3, there exist coprime integers $a$ and $b$ such that  \[x^\ell + y^mi = (a+bi)^{17} \quad \text{and} \quad z = a^2+b^2. \] Since $x$ is even, $a$ is even and $b$ is odd.

As before, we write $K = \Q(\zeta_{17} + \zeta_{17}^{-1})$ and follow the notation of the previous sections. We fix $j=1$ and $k=4$ so that $\theta_j$ and $\theta_k$ are interchanged by the unique involution in $\mathrm{Gal}(K / \Q)$. We write $K'$ for the unique degree $2$ subfield of $K$, with ring of integers $\mathcal{O}_{K'}$, and we write $\mathcal{B}_{17}$ for the unique prime of $K'$ above $17$.

\bigskip

\noindent \textbf{Case 1: $17 \nmid x$}

\bigskip

Let $u,v,$ and $w$ be defined as in (\ref{uvw}).
The Frey elliptic curve we define is
\[ F_1 : \; Y^2=X(X-u)(X+v) \, . \] 
The curve $F_1$ is defined over $K$, but not necessarily over $K'$.

\bigskip

\noindent \textbf{Case 2: $17 \mid x$}

\bigskip

Let
\begin{equation*}
u'= 
\frac{\beta_j}{(\theta_j-2)}, \qquad
v'=- \frac{\beta_k}{(\theta_k-2)}, \qquad
w'=\frac{4 (\theta_j-\theta_k)}{(\theta_j-2)(\theta_k-2)} \cdot a^2.
\end{equation*}
The Frey elliptic curve we define is 
\begin{equation*}
F_2 : \; Y^2=X(X-u')(X+v'). 
\end{equation*} By our choice of $j$ and $k$, the curve $F_2$ is defined over $K'$, and we view it as a curve defined over $K'$. This curve has a $2$-torsion point over $K'$ and will have full $2$-torsion over $K$, but it will not necessarily have full $2$-torsion over $K'$.

\begin{lemma}[{\citep[Lemma 6.1]{AnniSiksek}}]\label{llower}
Let $i = 1$ or $2$, so that $F_i$ is one of the Frey curves defined above. Suppose $\overline{\rho}_{F_i,\ell}$
is irreducible and $F_i$ is modular. Then
$\overline{\rho}_{F_i,\ell}\sim \overline{\rho}_{\mathfrak{f}_i,\lambda_i}$
for a Hilbert newform $\mathfrak{f}_i$ at level $\mathcal{N}_{\ell,i}$, where
\[
\mathcal{N}_{\ell,1}  = 2 \cdot \mathcal{O}_K, \qquad 
\mathcal{N}_{\ell,2}  = 2 \cdot \mathcal{B}_{17},
\]
and $\lambda_i \mid \ell$ is a prime of $\Q_{\mathfrak{f}_i}$.
\end{lemma}

The curves $F_1$ and $F_2$ are modular by \citep[Theorem 1.3]{recipes} (or by using the modularity results in \citep{AnniSiksek}). In order to apply this lemma, we must first prove the irreducibility of $\overline{\rho}_{F_i,\ell}$ for $i=1$ and $2$. Although we need only prove this for $\ell > 5$, we prove irreducibility for $\ell =5$ too, in the hope that our subsequent results may, in the future, be extended to include the case $\ell = 5$.

\begin{lemma}\label{irred17}
Let $i = 1$ or $2$, so that $F_i$ is one of the Frey curves defined above. Then $\overline{\rho}_{F_i,\ell}$ is irreducible for $\ell \geq 5$.
\end{lemma}

We first prove the following lemma.

\begin{lemma}\label{modlem} We have \begin{enumerate}[(i)]
\item $X_0(14)(K') = X_0(14)(\Q(\sqrt{17}))$.
\item $X_0(11)(K') = X_0(11)(\Q(\sqrt{17}))$.
\item $X_0(20)(K) = X_0(20)(\Q(\sqrt{17}))$.
\item Let $C$ be the elliptic curve with Cremona reference 52a1 given by $y^2 = x^3 + x - 10$. Then $C(K) = C(\Q) = \Z / 2\Z$.
\end{enumerate}
\end{lemma}

\begin{proof}
The curves $X_0(14)$ and $X_0(11)$ are elliptic curves, and it is straightforward to verify parts (i) and (ii) directly with \texttt{Magma}.

Next, let $X = X_0(20)$. This is an elliptic curve, given by Cremona label 20a1, and admits the following model over $\Q$: \[ X: \; y^2=x^3+x^2+4x+4.\]

The minimal polynomial of $\theta_1$ over $K'$ is a quadratic polynomial and we set $d$ to be its discriminant, so that $K = K'(\sqrt{d})$. We denote by $X_{d}$ the quadratic twist of $X$ by $d$. Then $X$ and $X_d$ are isomorphic over $K$, with an isomorphism given by \begin{align*} \varphi: X(K)  \longrightarrow X_d(K), \qquad  (x,y) \longmapsto \left( \frac{x}{d}, \frac{y}{d \sqrt{d}} \right).
\end{align*}

Using \texttt{Magma} we compute the following: \begin{align*} X(K') & = X(\Q(\sqrt{17})) = \Z \oplus \Z / 6\Z  = \langle R  \rangle \oplus \langle Q \rangle,\\ X_d(K') & = \Z / 2\Z, \\ X(K)_{\mathrm{tors}} & = X(\Q)_{\mathrm{tors}} = \Z / 6\Z =  \langle Q \rangle, \end{align*} where $R = \left((3\sqrt{17}+5)/8,(9\sqrt{17}+47)/16 \right)$ and $Q = (4,10)$. We were unable to directly compute $X(K)$ with \texttt{Magma}. However, we can start by noting  that  \[\mathrm{Rank}(X(K)) = \mathrm{Rank}(X(K')) + \mathrm{Rank}(X_d(K')) = 1.\] Next, let $P \in X(K)$ and let $\sigma \in \mathrm{Gal}(K / K')$. Then \[ P+P^\sigma \in X(K') \quad \text{ and} \quad \varphi(P-P^\sigma) \in X_d(K') = X_d(K')_{\mathrm{tors}} \, . \] Applying $\varphi^{-1}$ we have $P-P^\sigma \in X(K)_{\mathrm{tors}} = X(\Q)_{\mathrm{tors}}$. It follows that $2P  = (P+P^\sigma) + (P - P^\sigma) \in X(K')$. 

Now choose $P \in X(K)$ such that $X(K) =  \langle P \rangle \oplus \langle Q \rangle$, and write $R = rP+sQ$ for $r,s \in \Z$ with $0 \leq s \leq 5$. If $r = 2r'+1$ is odd, then \[P = R - sQ - r'(2P) \in X(K'). \] If $r = 2r'$ is even, then $R-sQ = 2(r'P)$. To obtain a contradiction, it will suffice to show that $R$ and $R+Q$ are not $2$-divisible. The prime $137$ is totally split in $K$. Let $\mathfrak{q}$ denote a prime of $K$ above $137$, and let $k = \mathcal{O}_K / \mathfrak{q}$. We find that $X(k) = \Z/2\Z \oplus \Z/60\Z$, and that the points $\tilde{R}$ and $\tilde{R}+\tilde{Q}$ both have order $60$; a contradiction in each case since there are no points of order $120$ in $X(k)$. We conclude that $P \in X(K')$, and thus $X(K) = X(K') = X(\Q(\sqrt{17}))$. This proves part (iii).

Finally, for part (iv), we first verify that $C(\Q) = C(K)_{\mathrm{tors}} = \Z / 2\Z$. Then defining $d$ as above, we check that \[ \mathrm{Rank}(C(K)) = \mathrm{Rank}(C(K')) + \mathrm{Rank}(C_d(K')) = 0, \] as required. 
\end{proof}

\begin{proof}[Proof of Lemma \ref{irred17}]
Suppose $\overline{\rho}_{F_i,\ell}$ is reducible. In Case 1, arguing as in \citep[p.~1165]{AnniSiksek}, we find that there exists an elliptic curve defined over $K$ with good reduction at the unique prime of $K$ above $17$, full $2$-torsion over $K$, and a torsion point of order $2\ell$ over $K$. By the Hasse--Weil bounds, we have \[\ell \leq \frac{(\sqrt{17}+1)^2}{4} <7, \] so $\ell = 5$.
In Case 2, again arguing as in \citep[p.~1165]{AnniSiksek}, we deduce the existence of an elliptic curve defined over $K'$ with a torsion point of order $2 \ell$ over $K'$. By \citep[Theorem~1.2]{DKSS}, the largest prime order of a point of an elliptic curve defined over a quartic field is $17$, and since $\ell \ne 17$, we obtain $5 \leq \ell \leq 13$.

It remains to deal with $\ell = 5$ in Case 1, and $\ell = 5, 7, 11,$ and $13$ in Case 2. When $\ell = 5$, the curves $F_1$ and $F_2$ give rise to non-cuspidal $K$-points on the modular curve $X_0(20)$. For $\ell = 7$, the curve $F_2$ gives rise to a non-cupsidal $K'$-point on $X_0(14)$, and when $\ell = 11$, the curve $F_2$ gives rise to a non-cupsidal $K'$-point on $X_0(11)$. Now, applying Lemma $\ref{modlem}$, we see that we in fact obtain $\Q(\sqrt{17})$-points on each of these three modular curves. It follows that $j(F_i) \in \Q(\sqrt{17})$ for $i \in\{1, 2\}$ when $\ell = 5$, and for $i = 2$ when $\ell = 7$ or $11$. 

Let $\hat{\mathfrak{q}}$ denote one of the two primes of $\Q(\sqrt{17})$ above $2$, and let $\mathfrak{q} = \hat{\mathfrak{q}}\mathcal{O}_K$, which is a prime of $K$ above $2$. Viewing $j(F_i) \in K$, for $i \in \{1,2\}$ we have $v_\mathfrak{q}(j(F_i)) = -(20v_2(a) -4) $, and we find that \[ 2^{20v_2(a)-4}j(F_i)\equiv \frac{ \theta_j^2\theta_k^2 }{(\theta_j-\theta_k)^2} \pmod{\mathfrak{q}}. \] We verify that $\frac{ \theta_j^2\theta_k^2 }{(\theta_j-\theta_k)^2} \pmod{\mathfrak{q}} \notin \mathbb{F}_2$ for $j=1$ and $k=4$ (in fact this holds for any choice of $1 \leq j < k \leq 8$). However, $\mathcal{O}_{ \Q(\sqrt{17})} / \hat{\mathfrak{q}} = \mathbb{F}_2$, contradicting $j(F_i) \in \Q(\sqrt{17})$.

Finally, we consider $\ell = 13$ for Case 2. Since $F_2$ has full $2$-torsion over $K$, it will give rise to a non-cuspidal $K$-point, which we denote $P$, on the modular curve $X_0(52)$. As in Lemma \ref{modlem}, we write $C$ for the elliptic curve with Cremona label 52a1. This is an optimal elliptic curve, and we have the modular parametrisation map defined over $\Q$: \[ \varphi: X_0(52) \longrightarrow C. \] 
The curve $C$ has modular degree $3$, so the degree of $\varphi$ is $3$. We have $\varphi(P) \in C(K) = C(\Q) \cong \Z /2 \Z$ by Lemma \ref{modlem}. So $P \in \varphi^{-1}(C(\Q))$, which has size at most $6$ since $\varphi$ has degree $3$. However, $X_0(52)$ has $6$ rational cusps, so $\varphi^{-1}(C(\Q))$ must consist of only cusps, contradicting the fact that $P$ is a non-cuspidal point. 
\end{proof}

We note that the idea of using the modular parametrisation map to study points on modular curves is present in the author's work in \citep[pp.~16--21]{MJ}. Here, we did not even need to compute a model for $X_0(52)$ to obtain the desired conclusion.

Having proven the necessary modularity and irreducibility statements, we can proceed to apply  Lemma \ref{llower} to the curves $F_1$ and $F_2$.

\begin{proof}[Proof of Theorem \ref{Thm3}] We suppose $(x,y,z)$ is a non-trivial primitive solution to (\ref{maineq}). Let $i=1$ or $2$ according to whether $17 \nmid x$ or $17 \mid x$, and let $F_i$ denote the Frey curve associated to this solution. We apply Lemma \ref{llower} to conclude that $\overline{\rho}_{F_i,\ell}\sim \overline{\rho}_{\mathfrak{f}_i,\lambda_i}$
for a Hilbert newform $\mathfrak{f}_i$ at level $\mathcal{N}_{\ell,i}$, where $\lambda_i \mid \ell$ is a prime of $\Q_{\mathfrak{f}_i}$.

The spaces of Hilbert cusp forms that are new at levels $\mathcal{N}_{\ell,1}$ and  $\mathcal{N}_{\ell,2}$ respectively have dimensions $647$ and $49$, and we can compute their newform decompositions using \texttt{Magma}. We use the same notation as in \citep[pp.~1166--1168]{AnniSiksek}, and follow the same method, when possible, to eliminate the newforms at these levels. In Case 2, we have $\mathcal{N}_{\ell,2} = 2 \cdot \mathcal{B}_{17}$ and using the set of primes $S = \{3, 67, 101\}$ in the sieve, we are able to eliminate all the newforms for primes $\ell >5$ with $\ell \ne 17$. In Case 1, there are $35$ newforms we would like to eliminate. Using the set of primes $S = \{3, 67, 101\}$ again, we eliminate $31$ of these newforms for all primes  $\ell >3$ with $\ell \ne 17$. 

The four remaining newforms, which we denote $\mathfrak{g}_1, \mathfrak{g}_2, \mathfrak{g}_3$, and $\mathfrak{g}_4$, have Hecke eigenfields of degree $136, 152, 152$, and $160$ respectively, and we are unable to compute their Hecke eigenvalues using \texttt{Magma}. However, for a prime $\mathfrak{q} \mid q$ of $K$ with $q \nmid 2 \cdot 17$, by considering the factorisation of the characteristic polynomial of the Hecke operator at $\mathfrak{q}$, we can compute the minimal polynomial of $a_\mathfrak{q}(\mathfrak{g}_i)$ for each $i$. In particular, this allows us to compute $\mathrm{Norm}_{\Q_{\mathfrak{g}_i} / \Q}(c - a_\mathfrak{q}(\mathfrak{g}_i))$ for any $c \in \Z$. Let \[  A_q \coloneqq \{0 \leq \eta, \mu \leq q-1, (\eta,\mu) \ne (0,0) \}. \] Then for $(\eta,\mu) \in A_q$, we can compute the quantity $\mathrm{Norm}_{\Q_{\mathfrak{g}_i} / \Q}(B_\mathfrak{q}(\mathfrak{g}_i, \eta, \mu))$, where $B_\mathfrak{q}(\mathfrak{g}_i, \eta, \mu)$ is defined as in  \citep[pp.~1167]{AnniSiksek}. We then have that $\ell \mid  B_\mathfrak{q}(\mathfrak{g}_i)$, where \[ B_\mathfrak{q}(\mathfrak{g}_i) \coloneqq q \prod_{(\eta,\mu) \in A_q}  \mathrm{Norm}_{\Q_{\mathfrak{g}_i} / \Q}(B_\mathfrak{q}(\mathfrak{g}_i, \eta, \mu)). \] Since this holds for each prime $\mathfrak{q} \mid q$ with $q \nmid 2 \cdot 17$, we can choose several such primes $\mathfrak{q}$ and compute the greatest common divisor of the values  $B_\mathfrak{q}(\mathfrak{g}_i)$. We choose one prime of $K$ above each of the rational primes in the set $\{3,67,157\}$ and compute the greatest common divisor of the values  $B_\mathfrak{q}(\mathfrak{g}_i)$ for each $i$. This greatest common divisor is not divisible by any prime $ >3$ when $i=1$, and is not divisible by any prime $>5$ for $i=2,3,$ and $4$. These computations complete the proof of the theorem.
\end{proof}

It would of course be preferable to eliminate the condition $\ell,m \ne 5$ in Theorem \ref{Thm3}. There are three obstructing newforms at level $\mathcal{N}_{\ell,1}$ (the newforms $\mathfrak{g}_2,\mathfrak{g}_3,$ and $\mathfrak{g}_4$) with  Hecke eigenfields of degree $152,152,$ and $160$. There are four obstructing newforms at level $\mathcal{N}_{\ell,2}$ with Hecke eigenfields of degree $2,2,6,$ and $6$. We expect that for each of these newforms $\mathfrak{f}$ the representation $\overline{\rho}_{\mathfrak{f},\lambda}$ is reducible, for $\lambda \mid 5$ a prime of $\Q_\mathfrak{f}$, and that this is why we are unable to discard them. Proving this would in fact allow us to discard these newforms since $\overline{\rho}_{F_i,5}$ is irreducible for $i=1$ and $2$. This seems like a difficult task, and would likely require an extension of the ideas present in \citep[5--8]{ext}. We note that it is also possible to twist the curve $F_1$ so that it is defined over the smaller field $K'$ (see \citep[p.~1158]{AnniSiksek}) and obtain a different set of Hilbert newforms. We were still unable to eliminate $\ell = 5$ in this case due to the presence of four obstructing newforms. 

\bibliographystyle{plainnat}

\Addresses

\end{document}